\documentclass[12pt]{amsart}
\usepackage{amsmath}
\usepackage{eucal}
\usepackage{amssymb}
\usepackage[all]{xy}
\usepackage{longtable}
\usepackage[osf,sc]{mathpazo}
\usepackage[titletoc]{appendix}
\usepackage{color}
\usepackage{multirow,bigdelim}
\usepackage{stmaryrd}
\usepackage{verbatim}
\usepackage{mathrsfs}
\usepackage{tikz-cd}

\allowdisplaybreaks

\newtheorem{theorem}[equation]{Theorem}
\newtheorem{lemma}[equation]{Lemma}
\newtheorem{proposition}[equation]{Proposition}
\newtheorem{corollary}[equation]{Corollary}

\newtheorem{definition}[equation]{Definition}

\theoremstyle{definition}
\newtheorem{example}[equation]{Example}
\theoremstyle{remark}
\newtheorem{remark}[equation]{Remark}

\numberwithin{equation}{subsection}

\allowdisplaybreaks[1]

\newcommand{\TT}{\mathbb{T}}

\newcommand{\CC}{\mathbb{C}}

\DeclareMathAlphabet{\matheur}{U}{eur}{m}{n}

 \DeclareMathOperator{\Lie}{Lie}
\DeclareMathOperator{\Ker}{Ker} \DeclareMathOperator{\GL}{GL}
\DeclareMathOperator{\Mat}{Mat} 
\DeclareMathOperator{\End}{End}

\DeclareMathOperator{\Id}{Id} \DeclareMathOperator{\im}{Im}

\newcommand{\ok}{\overline{k}}

\newcommand{\tr}{\mathrm{tr}}

\begin{document}

\title[]{{\large{O\MakeLowercase{n }F\MakeLowercase{urusho's analytic continuation of }D\MakeLowercase{rinfeld logarithms}}}}

\author{Yen-Tsung Chen}
\address{Department of Mathematics, National Tsing Hua University, Hsinchu City 30042, Taiwan R.O.C.}

\email{ytchen.math@gmail.com}


%    General info

\date{\today}

\begin{abstract} 
    In the present paper, we establish an analytic continuation of Drinfeld logarithms by using the techniques introduced in \cite{Fur20}. This result can be seen as an analogue of the analytic continuation of the elliptic integrals of the first kind for Drinfeld modules.
\end{abstract}


\maketitle

\tableofcontents

\section{Introduction}
\subsection{Motivation}
    Let $\Lambda\subset\mathbb{C}$ be a rank $2$ lattice, that is, a rank $2$ discrete free $\mathbb{Z}$-module. We consider the associated Weierstrass $\wp$-function
    \[
        \wp_\Lambda(z):=\frac{1}{z^2}+\sum_{0\neq\lambda\in\Lambda}\left(\frac{1}{(z-\lambda)^2}-\frac{1}{\lambda^2}\right).
    \]
    For all $z\in\mathbb{C}\setminus\Lambda$, the Weierstrass $\wp$-function and its derivative satisfy the cubic equation
    \[
        \wp_\Lambda'(z)^2=4\wp_\Lambda(z)^3-60G_4(\Lambda)\wp(z)-140G_6(\Lambda).
    \]
    Here $G_{2n}(\Lambda)$ for $n\in\mathbb{Z}_{>1}$ is the series
    \[
        G_{2n}(\Lambda):=\sum_{0\neq w\in\Lambda}w^{-2n}.
    \]
    Let $E_\Lambda$ be the elliptic curve defined by the following equation
    \[
        Y^2=4X^3-60G_4(\Lambda)X-140G_6(\Lambda).
    \]
    Then we have the following short exact sequence of abelian groups
    \begin{equation}\label{Eq:Elliptic_Curve}
        0\to \Lambda\hookrightarrow\mathbb{C}\overset{\exp_\Lambda(\cdot)}{\twoheadrightarrow}E_\Lambda(\mathbb{C})\to 0,
    \end{equation}
    where $\exp_\Lambda(\cdot)$ is defined by
    \[
        \exp_\Lambda(z)=\begin{cases}
            [\wp(z):\wp'(z):1],~&\mathrm{if}~z\not\in\Lambda\\
            [0:1:0],~&\mathrm{if}~z\in\Lambda.
        \end{cases}
    \]
    In particular, we obtain the complex analytic group isomorphism $\exp_\Lambda:\mathbb{C}/\Lambda\overset{\sim}{\to}E_\Lambda(\mathbb{C})$.
    
    One may naturally ask the converse problem. More precisely, let $E$ be a complex elliptic curve. How do we recover the corresponding period lattice $\Lambda_E$ and the complex analytic group isomorphism $\exp_{\Lambda_E}^{-1}:E(\mathbb{C})\to\mathbb{C}/\Lambda_E$? This question can be answered by the \emph{analytic continuation of an elliptic integral of the first kind}. In what follows, we follow \cite[Sec.VI.5-Sec.VI.8]{Kna92} closely (see also \cite[Sec.VI]{Sil09}) to give a brief exposition of the construction of $\exp_{\Lambda_E}^{-1}:E(\mathbb{C})\to\mathbb{C}/\Lambda_E$.
    
    After the linear change of variables, the defining equation of the elliptic curve $E$ can be written as the form
    \[
        Y^2=4(X-a)(X-b)(X-c)
    \]
    with pairwise distinct $a,b,c\in\mathbb{C}$ and $a+b+c=0$. Let $P=[x:y:1]\in E(\mathbb{C})$. Suppose that the elliptic curve $E$ is induced from a lattice $\Lambda_E$. Then we have $x=\wp_{\Lambda_E}(z)$ and
    \[
        \left(\frac{dx}{dz}\right)^2=4(x-a)(x-b)(x-c).
    \]
    In particular, we must have
    \begin{equation}\label{Eq:Elliptic_Integral}
        z=\int\frac{dx}{2\sqrt{(x-a)(x-b)(x-c)}}.
    \end{equation}
    Note that the right hand side of \eqref{Eq:Elliptic_Integral} is called an \emph{elliptic integral of the first kind}. 
    
    This observation suggests the potential candidate of $\exp_{\Lambda_E}^{-1}$. However, there are two issues about the right hand side of \eqref{Eq:Elliptic_Integral}. The first concern is that taking square roots depends on the choice of two branches. This issue can be solved by the analytic continuation of a fixed choice of the square root in a precise compactified double covering of $\mathbb{C}\setminus\{a,b,c\}$. The second concern is that the integral occurs in the right hand side of \eqref{Eq:Elliptic_Integral} is not path-independent. However, the values of this integral arising from different path are equal modulo some lattice $\Lambda_E$. In other words, it is well-defined in the quotient $\mathbb{C}/\Lambda_E$. In conclusion, we have the following theorem.
    \begin{theorem}[{cf. \cite[Prop.~5.2]{Sil09}}]
        Let $E$ be a complex elliptic curve. Then there is an explicitly constructed rank $2$ lattice $\Lambda_E\subset\mathbb{C}$ such that the map
        \begin{align*}
            F:E(\mathbb{C})&\to\mathbb{C}/\Lambda_E\\
            P&\mapsto\int_\mathcal{O}^P\omega~(\mbox{mod}~\Lambda_E)
        \end{align*}
        is a complex analytic group isomorphism, where $\mathcal{O}=[0:1:0]$ is the point at infinity, and $\omega=dx/y$ is a holomorphic differential form on $E(\mathbb{C})$. Moreover, we have $F=\exp_{\Lambda_E}^{-1}$.
    \end{theorem}
    
    The main purpose of this paper is to study an analogous construction of $\exp_{\Lambda_E}^{-1}$ for Drinfeld modules using the theory of Anderson dual $t$-motives together with the techniques of the analytic continuation method developed by Furusho in \cite{Fur20}.
    
\subsection{The main result}
    Let $\mathbb{F}_q$ be the finite field with $q$ elements, for $q$ a power of a prime number $p$. Let $\mathbb{P}^1_{/\mathbb{F}_q}$ be the projective line over $\mathbb{F}_q$ with a fixed point at infinity $\infty$. Let $k=\mathbb{F}_q(\theta)$ be the functoin field of $\mathbb{P}^1_{/\mathbb{F}_q}$ and $A=\mathbb{F}_q[\theta]$ be the ring of rational functions on $\mathbb{P}^1_{/\mathbb{F}_q}$ regular away from $\infty$. We denote by $|\cdot|_\infty$ the normalized $\infty$-adic norm on $k$ so that $|\theta|_\infty=q$. Let $k_\infty$ be the completion of $k$ with respect to $|\cdot|_\infty$ and let $\mathbb{C}_\infty$ be the $\infty$-adic completion of a fixed algebraic closure of $k_\infty$. Let $\ok$ be the algebraic closure of $k$ in $\mathbb{C}_\infty$.
    
    Let $\Lambda\subset\mathbb{C}_\infty$ be a free $A$-module of rank $r\in\mathbb{Z}_{>0}$ so that for any $\nu\in\mathbb{R}_{>0}$ the intersection of $\Lambda$ and $\{z\in\mathbb{C}_\infty\mid|z|_\infty<\nu\}$ is finite. We can associate the following $\mathbb{F}_q$-linear power series for $\Lambda$:
    \begin{equation}\label{Eq:Expoential_of_Lattice}
        \exp_\Lambda(X):=X\prod_{0\neq\lambda\in\Lambda}\left(1-\frac{X}{\lambda}\right)\in\mathbb{C}_\infty\llbracket X\rrbracket.
    \end{equation}
    Note that $\exp_\Lambda(X)$ induces a surjective, entire, $\mathbb{F}_q$-linear function on $\mathbb{C}_\infty$, from which we have an analogue of \eqref{Eq:Elliptic_Curve} as $\mathbb{F}_q[t]$-modules
    \begin{equation}\label{Eq:Drinfeld_Modules_Uniformization}
        0\to \Lambda\hookrightarrow\mathbb{C}_\infty\overset{\exp_\Lambda(\cdot)}{\twoheadrightarrow}\mathbb{C}_\infty\to 0,
    \end{equation}
    where the $\mathbb{F}_q[t]$-module structure is explained as follows.
    
    Let $t$ be a variable which is independent from $\theta$. There is an $\mathbb{F}_q$-algebra homomorphism $\phi^\Lambda:\mathbb{F}_q[t]\to\End_{\mathbb{F}_q}(\mathbb{G}_a)$ so that for any $z\in\mathbb{C}_\infty$ and $a\in\mathbb{F}_q[t]$
    \[
        a(\theta)\exp_\Lambda(z)=\exp_\Lambda(\phi^\Lambda_a(z)).
    \]
    Here $\phi^\Lambda_a$ refers to the image of $a$ under $\phi^\Lambda$, and $\End_{\mathbb{F}_q}(\mathbb{G}_a)$ is the ring of $\mathbb{F}_q$-linear endomorphisms on the additive group $\mathbb{G}_a$. Note that $\phi^\Lambda$ gives a non-scalar $\mathbb{F}_q[t]$-action on $\mathbb{C}_\infty$. Thus, \eqref{Eq:Drinfeld_Modules_Uniformization} is a short exact sequence of $\mathbb{F}_q[t]$-modules, where the middle $\mathbb{C}_\infty$ is the $\mathbb{F}_q[t]$-module with the scalar action replacing $t$ by $\theta$. The pair $E_\Lambda:=(\mathbb{G}_a,\phi^\Lambda)$ is called a \emph{Drinfeld module} of rank $r$, which was introduced by Drinfeld in his seminal work \cite{Dri74}. The case $r=1$ was initiated by Carlitz in \cite{Car35} and it serves as an analogue of the multiplicative group $\mathbb{G}_m$ in the global function fields setting.
    
    We denote by $E_\Lambda(\mathbb{C}_\infty)$ the $\mathbb{F}_q[t]$-module whose underlying space is the $\mathbb{C}_\infty$-valued points of the additive group $\mathbb{G}_{a}$, and the $\mathbb{F}_q[t]$-module structure arises from $\phi^\Lambda$. Then $\eqref{Eq:Drinfeld_Modules_Uniformization}$ induces an analytic isomorphism of $\mathbb{F}_q[t]$-modules $\exp_\Lambda:\mathbb{C}_\infty/\Lambda\overset{\sim}{\to}E_\Lambda(\mathbb{C}_\infty)$. Inspired by the situation of complex elliptic curves, let $E=(\mathbb{G}_a,\phi)$ be a Drinfeld module of rank $r$. The first question is how to recover the corresponding period lattice $\Lambda_E$. If we set $\exp_E(\cdot):=\exp_{\Lambda_E}(\cdot)$, then the further question is how to recover the analytic isomorphism of $\mathbb{F}_q[t]$-modules $\exp_E^{-1}:E(\mathbb{C}_\infty)\to\mathbb{C}_\infty/\Lambda_E$. It is known by Anderson that specializations of rigid analytic trivializations associated to dual $t$-motives of Drinfeld modules can be used to recover their periods (see \eqref{Eq:E_0_Periods}). In the present paper, we adopt the rigid analytic trivialization coming from dual $t$-motives. Note that there is a dual notion of the rigid analytic trivialization arising from the associated $t$-motives of Drinfeld modules (see \cite{And86}), which also supplies the periods of Drinfeld modules (see \cite[Sec.~2.4]{Gos94}, \cite[Sec.~4]{Pel08}), and a direct way of its construction without using periods as initial inputs is recently established by Khaochim and Papanikolas in \cite{KP21}.
    
    Let $\log_{E}(X)\in\mathbb{C}_\infty\llbracket X\rrbracket$ be the formal inverse of the power series $\exp_{E}(X)$. We refer $\log_{E}$ as the \emph{logarithmic function} of the Drinfeld module $E$. As a formal power series, it satisfies
    \[
        a(\theta)\log_{E}(X)=\log_{E}(\phi_a(X))
    \]
    for any $a\in\mathbb{F}_q[t]$. In general, $\log_{E}(\cdot)$ is only defined on a small region in $\mathbb{C}_\infty$. Note that the radius of convergence $R_E$ of $\log_{E}(\cdot)$ is determined explicitly in \cite[Cor.~4.5]{KP21}. It is natural to ask whether one can construct an extension of $\log_{E}(\cdot)$ so that it serves as the analytic isomorphism of $\mathbb{F}_q[t]$-modules $\exp_{E}^{-1}:E(\mathbb{C}_\infty)\to\mathbb{C}_\infty/\Lambda_E$. The present paper is to give an affirmative answer to this question, which can be summarized as the following theorem (see Theorem~\ref{Thm:Main_Thm_1}, Theorem~\ref{Thm:Main_Thm_2}, and Corollary~\ref{Thm:Main_Thm_3}).
    
    \begin{theorem}\label{Thm:Main_Thm_Intro}
        Let $E$ be a Drinfeld module defined over $\mathbb{C}_\infty$ with period lattice $\Lambda_E\subset\mathbb{C}_\infty$. Then there is an explicitly constructed function $\overset{\rightarrow}{\log_{E}}:E(\mathbb{C}_\infty)\to\mathbb{C}_\infty/\Lambda_E$ allowing analytic lifts locally such that the following assertions hold.
        \begin{enumerate}
            \item If $\xi\in E(\mathbb{C}_\infty)$ with $|\xi|_\infty<R_E$, then
            \[
                \overset{\rightarrow}{\log_E}(\xi)\equiv\log_E(\xi)~(\mbox{mod}~\Lambda_E).
            \]
            \item 
            The function $\overset{\rightarrow}{\log_{E}}$ is an isomorphism of $\mathbb{F}_q[t]$-modules. More precisely, for any $\xi\in\mathbb{C}_\infty$, we have
            \[
                \overset{\rightarrow}{\log_E}(\phi_t(\xi))\equiv\theta\overset{\rightarrow}{\log_E}(\xi) ~(\mathrm{mod}~\Lambda_E).
            \]
            \item The function $\overset{\rightarrow}{\log_{E}}$ is the inverse of $\exp_{\Lambda_E}:\mathbb{C}_\infty/\Lambda_E\to E(\mathbb{C}_\infty)$.
        \end{enumerate}
    \end{theorem}

\subsection{Strategy and organization}
    The key ingredients of our strategy can be divided into two steps. In the first step, we establish an explicit formula for the deformation series of Drinfeld logarithms using the Anderson dual $t$-motives associated to the given Drinfeld modules. After an appropriate specialization, we deduce a formula for the coefficients of Drinfeld logarithms. In \cite{EGP13,EGP14}, El-Guindy and Papanikolas constructed the deformation series of Drinfeld logarithms using combinatorial tools such as shadow partitions. Our results are inspired by the calculations in \cite{KP21} and provide a different perspective of the construction of El-Guindy and Papanikolas. In fact, our formula is related to a more general result of Anderson's exponentiation theorem for abelian uniformizable $t$-modules (see Remark~\ref{Rem:Comparision_Anderson}), and it can be also seen as an explicit example of the abstract results established in \cite[Section~2]{ANDTR20} and \cite{Gre22} for the logarithms of abelian $t$-modules (see Remark~\ref{Rem:Comparision_ANDTR}). However, we emphasize that such explicit calculations are essential for our purpose to fit into the framework of Furusho's analytic continuation.
    
    In the second step, we follow the ideas of Furusho in \cite{Fur20} to extend the region of the convergence of the Drinfeld logarithms in question. The crucial point in Furusho's analytic continuation is to use the Artin-Schreier equations, which plays a substitute role of the iterated path integrals. One of the main difficulty in this step is to show that our extended Drinfeld logarithm is an isomorphism of $\mathbb{F}_q[t]$-modules. We overcome this issue by using a specific choice of a rigid analytic trivialization of the Drinfeld module in question as well as the explicit formula we obtained in the first step. Note that such precise information is crucial for our approach and is one of the obstacle to extending the analytic continuation to higher dimensional $t$-modules. Further investigations would be necessary to this direction.
    
    The organization of the present paper is given as follows. In Section~2, we set up our notation, and then we follow closely the exposition of \cite{NP21} to review the theory of Anderson dual $t$-motives associated to Drinfeld modules. For the related contents, readers can also consult \cite{BP20,HJ20}. In Section~3, we first use a specific $t$-frame of the Anderson dual $t$-motive to construct the deformation series of Drinfeld logarithms. Note that the resulting object coincides with the one constructed by El-Guindy and Papanikolas, while we adopt a different approach. Finally, we apply the ideas of Furusho in \cite{Fur20} to construct the analytic continuation of Drinfeld logarithms and establish the essential properties.

\begin{comment}
\section*{Acknowledgement}
    The author would like to thank M. Papanikolas for fruitful comments and discussions. The author is also grateful to C.-Y. Chang, H. Furusho, and O. Gezmis for many valuable suggestions on this project. Parts of this work were done when the author visited to Texas A\&M University, the author also thanks for their hospitality.
\end{comment}

\section{Preliminaries}
\subsection{Notation}
    \begin{itemize}
		\setlength{\leftskip}{0.8cm}
        \setlength{\baselineskip}{18pt}
		\item[$\mathbb{F}_q$] :=  A fixed finite field with $q$ elements, for $q$ a power of a prime number $p$.
		\item[$\infty$] :=  A fixed closed point on the projective line $\mathbb{P}^1(\mathbb{F}_q)$.
		\item[$A$] :=  $\mathbb{F}_q[\theta]$, the regular functions of $\mathbb{P}^1$ away from $\infty$.
		\item[$k$] :=  $\mathbb{F}_q(\theta)$, the function field of $\mathbb{P}^1$.
		\item[$k_\infty$] :=  The completion of $k$ at the place $\infty$.
		\item[$\mathbb{C}_\infty$] :=  The completion of a fixed algebraic closure of $k_\infty$.
        \item[$\overline{k}$] := The algebraic closure of $k$ inside $\mathbb{C}_\infty$.
    \end{itemize}
    
    Recall that $t$ is a variable independent from $\theta$. We set $\mathbb{T}$ to be the Tate algebra of rigid analytic functions converging on the closed unit disc in $\mathbb{C}_\infty$. More precisely, it is given by
    \[
        \mathbb{T}:=\{\sum_{i\geq 0}c_it^i\in\mathbb{C}_\infty\llbracket t\rrbracket\mid\lim_{i\to 0}|c_i|_\infty=0\}.
    \]
    For $\alpha\in\mathbb{C}_\infty^\times$, we further define
    \[
        \mathbb{T}_\alpha:=\{\sum_{i\geq 0}c_it^i\in\mathbb{C}_\infty\llbracket t\rrbracket\mid\lim_{i\to 0}|c_i|_\infty|\alpha|^i_\infty=0\}.
    \]
    Let $\mathbb{E}\subset\mathbb{T}$ be the ring of power series inducing an entire function on $\mathbb{C}_\infty$, and whose coefficients lie in a finite extension over $k_\infty$. Note that for $\alpha\in\mathbb{C}_\infty^\times$ we define a norm $||\cdot||_\alpha$ on $\mathbb{T}_\alpha$ given by
    \[
        ||f||_\alpha:=\sup_{i\geq 0}\{|c_i|_\infty|\alpha|_\infty^i\}<\infty
    \]
    whenever $f=\sum_{i\geq 0}c_it^i\in\mathbb{T}_\alpha$. Furthermore, the norm $||\cdot||_\alpha$ is extended to $\Mat_{m\times n}(\TT_\alpha)$ given by
    \[
        ||F||_\alpha:=\sup_{i,j}\{||f_{ij}||_\alpha\}<\infty
    \]
    whenever $F=(f_{ij})\in\Mat_{m\times n}(\TT_\alpha)$. By abuse of the notation, the norm on $\Mat_{m\times n}(\TT)$ will be denoted by $||\cdot||:=||\cdot||_1$.
	
\subsection{Drinfeld modules and Anderson dual $t$-motives}
    Let $R$ be an $A$-algebra containing $k$ and $\tau:=(x\mapsto x^q):R\to R$ be the Frobenius $q$-th power operator. We define $R[\tau]$ to be the twisted polynomial ring in $\tau$ over $R$ subject to the relation $\tau\alpha=\alpha^q\tau$ for $\alpha\in R$. A Drinfeld $\mathbb{F}_q[t]$-module over $R$ is a pair $E=(\mathbb{G}_{a/R},\phi)$ where $\mathbb{G}_{a/R}$ is the additive group scheme defined over $R$ and $\phi$ is an $\mathbb{F}_q$-algebra homomorphism $\phi:\mathbb{F}_q[t]\to R[\tau]$ so that $\im(\phi)\not\subset R$, and  $\phi_t=\theta+\kappa_{1}\tau+\cdots+\kappa_{r}\tau^r\in R[\tau]$ with $\kappa_r\neq 0$. We define $r=\deg_\tau\phi_t$ to be the rank of $E$, and we define $\partial\phi_a:=a(\theta)$ for any $a\in\mathbb{F}_q[t]$.
    
    Let $E=(\mathbb{G}_{a/\mathbb{C}_\infty},\phi)$ be a Drinfeld $\mathbb{F}_q[t]$-module of rank $r$ defined over $\mathbb{C}_\infty$. Then there exists a unique $\mathbb{F}_q$-linear power series $\exp_E(X)\in\mathbb{C}_\infty\llbracket X\rrbracket$ so that
    \[
        \exp_E(X)\equiv X~(\mbox{mod}~X^q\mathbb{C}_\infty\llbracket X\rrbracket)
    \]
    and satisfying $\exp_E(a(\theta))=\phi_a(\exp_E(X))$ for any $a\in\mathbb{F}_q[t]$. Furthermore, $\exp_E(X)$ induces an entire, surjective, and $\mathbb{F}_q$-linear function on $\mathbb{C}_\infty$. One can associate an $A$-lattice $\Lambda_E:=\Ker(\exp_E(\cdot))\subset\mathbb{C}_\infty$ of rank $r$ for $E$, i.e. a free $A$-module of rank $r$ in $\mathbb{C}_\infty$ so that for any $\nu\in\mathbb{R}_{>0}$ the intersection of $\Lambda_E$ and $\{z\in\mathbb{C}_\infty\mid|z|_\infty<\nu\}$ is finite. Each non-zero element in $\Lambda_E$ is called a \emph{period} of $E$.
    %Recall that the \emph{exponential function} of the Drinfeld $\mathbb{F}_q[t]$-module $E$ is defined in \eqref{Eq:Expoential_of_Lattice} given by the power series $\exp_E(z)=\exp_{\Lambda_E}(z)=z\prod_{0\neq\lambda\in\Lambda_E}(1-z/\lambda)$. It satisfies that $\exp_E(a(\theta)z):=\phi_a(\exp(z))$ for all $a\in\mathbb{F}_q[t]$. The \emph{logarithmic function} $\log_E$ of the Drinfeld $\mathbb{F}_q[t]$-module $E$ is defined by the formal inverse of $\exp_E$. It satisfies that $a(\theta)\log_E(z)=\log_E(\phi_a(z))$ for all $a\in\mathbb{F}_q[t]$. In general, $\log_E$ is only defined on a small region in $\mathbb{C}_\infty$ even though $\exp_E$ defines an entire function on $\mathbb{C}_\infty$.
    
    In what follows, we recall the notion of dual $t$-motives by following \cite[Def.~4.4.1]{ABP04}. For $n\in\mathbb{Z}$, we define the $n$-fold Frobenius twisting
    \begin{align*}
	    \mathbb{C}_{\infty}((t))&\rightarrow\mathbb{C}_{\infty}((t))\\
	    f:=\sum_{i}a_it^i&\mapsto \sum_{i}a_i^{q^{n}}t^i=:f^{(n)}.
    \end{align*}
    It can be naturally extended to $\Mat_{m\times \ell}(\mathbb{C}_\infty((t)))$ entrywise by setting $B^{(n)}:=(B_{ij}^{(n)})$ whenever $B=(B_{ij})\in\Mat_{m\times\ell}(\mathbb{C}_\infty((t)))$. We denote by $\mathbb{C}_\infty[t,\sigma]$ the non-commutative $\mathbb{C}_\infty[t]$-algebra generated by $\sigma$ subject to the following relation:
    \[
        \sigma f=f^{(-1)}\sigma, \quad f\in\mathbb{C}_\infty[t].
    \]
    Note that $\mathbb{C}_\infty[t,\sigma]$ contains $\mathbb{C}_\infty[t]$ and $\mathbb{C}_\infty[\sigma]$. Moreover, its center is $\mathbb{F}_q[t]$. 

    \begin{definition}\label{Def:DtM}
        A dual $t$-motive is a left  $\mathbb{C}_\infty[t,\sigma]$-module $\mathcal{M}$ satisfying the following conditions:
        \begin{itemize}
            \item[(i)] $\mathcal{M}$ is a free left $\mathbb{C}_\infty[t]$-module of finite rank.
            \item[(ii)] $\mathcal{M}$ is a free left $\mathbb{C}_\infty[\sigma]$-module of finite rank.
            \item[(iii)] $(t-\theta)^n \mathcal{M}\subset \sigma \mathcal{M}$ for any sufficiently large integer $n$.
        \end{itemize}
    \end{definition}
    
    Let $E=(\mathbb{G}_{a/\CC_\infty},\phi)$ be a Drinfeld $\mathbb{F}_q[t]$-module of rank $r$ defined over $\CC_\infty$ with
    \[
        \phi_t=\theta+\kappa_1\tau+\cdots+\kappa_r\tau^r\in\CC_\infty[\tau].
    \]
    We associate a left $\CC_\infty[t,\sigma]$-module $\mathcal{M}_E$ for $E$ as follows: Let $\mathcal{M}_E:=\CC_\infty[\sigma]$. It naturally has a left $\CC_\infty[\sigma]$-module structure. The left $\CC_\infty[t]$-module structure of $\mathcal{M}_E$ is uniquely determined by the $t$-action, and it is given by
    \[
        tm:=m\phi_t^\star:=m\left(\theta+\kappa_1^{(-1)}\sigma+\cdots+\kappa_r^{(-r)}\sigma^r\right).
    \]
    It is clear that $\mathcal{M}_E$ is free of rank $1$ over $\CC_\infty[\sigma]$ and it is a straightforward computation that
    \[
        (t-\theta)\mathcal{M}_E\subset\sigma\mathcal{M}_E.
    \]
    Note that $\mathcal{M}_E$ is free of rank $r$ over $\mathbb{C}_\infty[t]$, and $\{1,\sigma,\dots,\sigma^{r-1}\}$ forms a $\CC_\infty[t]$-basis of $\mathcal{M}_E$. Hence $\mathcal{M}_E$ defines a dual $t$-motive. For the convenience of later use, we refer $\mathcal{M}_E$ as \emph{the associated dual $t$-motive of $E$}. 
    
    Now we are going to explain how to recover the Drinfeld module $E$ from its associated dual $t$-motive $\mathcal{M}_E$. Let $m=\sum_{i=0}^{n}\alpha_i\sigma^i\in\mathcal{M}_E$ with $\alpha_i\in\CC_\infty$. We define
    \[
        \epsilon_0(m):=\alpha_0\in\CC_\infty,~
        \epsilon_1(m):=\sum_{i=0}^n\alpha_i^{(i)}\in\CC_\infty
    \]
    Note that $\epsilon_0:\mathcal{M}_E\to\mathbb{C}_\infty$ is a $\mathbb{C}_\infty$-linear map and $\epsilon_1:\mathcal{M}_E\to\mathbb{C}_\infty$ is an $\mathbb{F}_q$-linear map. Moreover, due to Anderson (see also {\cite[Prop.~2.5.8]{HJ20}} or {\cite[Lem.~3.1.2]{NP21}}), $\epsilon_0$ and $\epsilon_1$ induce isomorphisms:
    \[
        \epsilon_0:\mathcal{M}_E/\sigma\mathcal{M}_E\cong \Lie(E)(\mathbb{C}_\infty),~
        \epsilon_1:\mathcal{M}_E/(\sigma-1)\mathcal{M}_E\cong
        E(\mathbb{C}_\infty)
    \]
    where $\epsilon_0$ is $\mathbb{C}_\infty[t]$-linear and $\epsilon_1$ is $\mathbb{F}_q[t]$-linear.

\begin{comment}
    We have the following lemma due to Anderson.
    
    \begin{lemma}[Anderson;~see also {\cite[Prop.~2.5.8]{HJ20}} or {\cite[Lem.~3.1.2]{NP21}}]\label{Lem:DtMotives_to_tModules}
        For any $a\in\mathbb{F}_q[t]$, we have the following commutative diagrams with exact rows:
        \[
        \begin{tikzcd}
            0 \arrow[r] & \mathcal{M}_E \arrow[d, "a(\cdot)"] \arrow[r, "\sigma(\cdot)"] & \mathcal{M}_E \arrow[d, "a(\cdot)"] \arrow[r, "\epsilon_0"] & \mathbb{C}_\infty \arrow[d, "\partial\phi_a(\cdot)"] \arrow[r] & 0 \\
            0 \arrow[r] & \mathcal{M}_E \arrow[r, "\sigma(\cdot)"] & \mathcal{M}_E \arrow[r, "\epsilon_0"] & \mathbb{C}_\infty \arrow[r] & 0
        \end{tikzcd}
        \]
        and
        \[
        \begin{tikzcd}
            0 \arrow[r] & \mathcal{M}_E \arrow[d, "a(\cdot)"] \arrow[r, "(\sigma-1)(\cdot)"] & \mathcal{M}_E \arrow[d, "a(\cdot)"] \arrow[r, "\epsilon_1"] & \mathbb{C}_\infty \arrow[d, "\phi_a(\cdot)"] \arrow[r] & 0 \\
            0 \arrow[r] & \mathcal{M}_E \arrow[r, "(\sigma-1)(\cdot)"] & \mathcal{M}_E \arrow[r, "\epsilon_1"] & \mathbb{C}_\infty \arrow[r] & 0.
        \end{tikzcd}
        \]
        In particular, $\epsilon_0$ and $\epsilon_1$ induce isomorphisms:
        \[
            \epsilon_0:\mathcal{M}_E/\sigma\mathcal{M}_E\cong \Lie(E)(\mathbb{C}_\infty),~
            \epsilon_1:\mathcal{M}_E/(\sigma-1)\mathcal{M}_E\cong E(\mathbb{C}_\infty)
        \]
        where $\epsilon_0$ is $\mathbb{C}_\infty[t]$-linear and $\epsilon_1$ is $\mathbb{F}_q[t]$-linear.
    \end{lemma}
\end{comment}
    
    Now we recall the notion of \emph{$t$-frame} based on the ideas of Anderson. Let $\{m_1,\dots,m_r\}$ be a $\mathbb{C}_\infty[t]$-basis of $\mathcal{M}_E$. Then there is an unique matrix $\Phi_E\in\Mat_r(\mathbb{C}_\infty[t])$ such that
    \[
        \sigma(m_1,\dots,m_r)^\tr=\Phi_E(m_1,\dots,m_r)^\tr.
    \]
    We define the following $\mathbb{C}_\infty[t]$-linear map which is an isomorphism:
    \begin{align*}
        \iota:\Mat_{1\times r}(\mathbb{C}_\infty[t])&\to\mathcal{M}_E\\
        (a_1,\dots,a_r)&\mapsto a_1m_1+\cdots+a_rm_r.
    \end{align*}
    The pair $(\iota,\Phi_E)$ is called a $t$-frame of $E$. Note that for any $(a_1,\dots,a_r)\in\Mat_{1\times r}(\mathbb{C}_\infty[t])$, we have
    \begin{align*}
        \sigma\iota(a_1,\dots,a_r)&=\sigma(a_1,\dots,a_r)(m_1,\dots,m_r)^\tr\\
        &=(a_1^{(-1)},\dots,a_r^{(-1)})\sigma(m_1,\dots,m_r)^\tr\\
        &=(a_1^{(-1)},\dots,a_r^{(-1)})\Phi_E(m_1,\dots,m_r)^\tr\\
        &=\iota((a_1^{(-1)},\dots,a_r^{(-1)})\Phi_E).
    \end{align*}
    
    Recall from our notation that
    \[
        \mathbb{T}_\theta=\{f=\sum_{i\geq 0}a_it^i\in\mathbb{C}_\infty\llbracket t\rrbracket\mid\lim_{i\to 0}q^i|a_i|_\infty=0\}.
    \]
    For $f=\sum_{i\geq 0}a_it^i\in\mathbb{T}_\theta$, we have defined the Gauss norm
    \[
        ||f||_\theta:=\sup_{i\geq 0}\{q^i|a_i|_\infty\}<\infty.
    \]
    Then we have the following result due to Anderson.
    \begin{lemma}[Anderson;~see also {\cite[Prop.~2.5.8]{HJ20}} or {\cite[Lem.~3.4.1]{NP21}}]
        There exists a unique bounded $\mathbb{C}_\infty$-linear map
        \[
            \mathscr{E}_0:\left(\Mat_{1\times r}(\mathbb{T}_\theta),||\cdot||_\theta\right)\to\left(\mathbb{C}_\infty,|\cdot|_\infty\right)
        \]
        of normed vector spaces such that $\mathscr{E}_0\mid_{\Mat_{1\times r}(\mathbb{C}_\infty[t])}=\epsilon_0\circ\iota$.
    \end{lemma}
    
    Since $\{1,\sigma,\dots,\sigma^{r-1}\}$ forms a $\mathbb{C}_\infty[t]$-basis of $\mathcal{M}_E$, we have the associated $t$-frame $(\iota,\Phi_E)$, where $\Phi_E\in\Mat_{r}(\mathbb{C}_\infty[t])$ is given by
    \begin{equation}\label{Eq:T_Frame_Of_Drinfeld_Modules}
        \Phi_E=\begin{pmatrix}
            0 & 1 & & \\
            \vdots & & \ddots \\
            0 & & & 1\\
            \frac{t-\theta}{\kappa_r^{(-r)}} & \frac{-\kappa_1^{(-1)}}{\kappa_r^{(-r)}} & \cdots &
            \frac{-\kappa_{r-1}^{(-r+1)}}{\kappa_r^{(-r)}} 
        \end{pmatrix}\in\Mat_{r}(\mathbb{C}_\infty[t]).
    \end{equation}
    A matrix $\Psi_E\in\GL_r(\mathbb{T})$ satisfying
    \begin{equation}\label{Eq:Difference_Equation}
        \Psi_E^{(-1)}=\Phi_E\Psi_E.
    \end{equation}
    is called a \emph{rigid analytic trivialization} of $E$. An explicitly construction of $\Psi_E$ is given in \cite[Sec.~3.4]{CP12} by following the method of Pellarin \cite[Sec.~4.2]{Pel08}. According to \cite[Prop.~3.1.3]{ABP04}, we have $\Psi_E\in\GL_r(\mathbb{T})\cap\Mat_r(\mathbb{E})$. In fact, by the argument given in \cite[Def.~3.4.5]{NP21}, we can even conclude that $\Psi_E\in\GL_r(\TT_\theta)$. It is known by Anderson (see also {\cite[Cor.~2.5.23,~Thm.~2.5.32]{HJ20}}, {\cite[Thm.~3.4.7]{NP21}}) that applying $\mathscr{E}_0$ to $\Mat_{1\times r}(\mathbb{F}_q[t])\Psi_E^{-1}$ can be used to recover the periods of $E$. More precisely, we have
    \begin{equation}\label{Eq:E_0_Periods}
        \Lambda_E=\mathscr{E}_0(\Mat_{1\times r}(\mathbb{F}_q[t])\Psi_E^{-1})\subset\mathbb{C}_\infty.
    \end{equation}

\section{Analytic continuation of the Drinfeld logarithms}
    In this section, we first give the detailed calculations for constructing the deformation series of Drinfeld logarithms using a specific choice of a $t$-frame (see \eqref{Eq:T_Frame_Of_Drinfeld_Modules}). Then we apply Furusho's method \cite{Fur20} to obtain an analytic continuation of Drinfeld logarithms.

\subsection{The deformation series of Drinfeld logarithms}
    Let $E=(\mathbb{G}_{a/\mathbb{C}_\infty},\phi)$ be a Drinfeld $\mathbb{F}_q[t]$-module of rank $r$ defined over $\mathbb{C}_\infty$ with
    \[
        \phi_t=\theta+\kappa_1\tau+\cdots+\kappa_r\tau^r\in \mathbb{C}_\infty[\tau].
    \]
    Recall that the logarithm map $\log_E$ is the formal inverse of the exponential map $\exp_E$, and $\log_E$ is only defined in a small region in $\mathbb{C}_\infty$. Let $R_E$ be the radius of convergence of the Drinfeld logarithm $\log_E$. We mention that $R_E$ has been determined explicitly in \cite[(2.11),~Cor.~4.5]{KP21}.
    Let $\xi\in\mathbb{C}_\infty$ with $|\xi|_\infty<R_E$. By \cite{EGP14}, a deformation series of $\log_E(\xi)$ is constructed explicitly
    %It is known in  that there is a deformation series $\mathcal{L}_E(\xi;t)\in\TT_\theta$ satisfying that
    %\begin{equation}\label{Eq:Specialization}
    %    \mathcal{L}_E(\xi;t)\mid_{t=\theta}=\log_E(\xi).
    %\end{equation}
    %The deformation series $\mathcal{L}_E(\xi;t)$ is constructed by 
    using the \emph{shadowed partitions} introduced in \cite{EGP13}. More precisely, for each $n,m\in\mathbb{Z}_{>0}$, we let $P_m(n)$ be the collection of $m$-tuples $(S_1,\dots,S_m)$ such that $S_i\subset\{0,1,\dots,n-1\}$ for each $1\leq i\leq m$ and the set $\{S_i+j\mid1\leq i\leq m,~0\leq j\leq i-1\}$ forms a partition of $\{0,\dots,n-1\}$. Here we set $P_r(0)=\{(\emptyset,\dots,\emptyset)\}$ to be the singleton set containing the $r$-tuple of empty sets. Then, for each non-negative integer $n$, we define
    \begin{equation}\label{Eq:Coe_of_Deformation_of_Log}
        \mathcal{B}_n(t):=\sum_{(S_1,\dots,S_r)\in P_r(n)}\prod_{i=1}^r\prod_{j\in S_i}\frac{\kappa_i^{q^j}}{t-\theta^{q^{i+j}}}\in\mathbf{C}_\infty(t).
    \end{equation}
    Here we adopt the convention that $\mathcal{B}_0(t)=1$.
    The deformation series of $\log_E(\xi)$ is given by
    \begin{equation}\label{Eq:Deformation_of_Log}
        \mathcal{L}_E(\xi;t):=\sum_{n=0}^\infty\mathcal{B}_n(t)\xi^{q^n}\in\TT_\theta.
    \end{equation}
    It is known in \cite[Thm~.6.13.(b)]{EGP14} that
    \begin{equation}\label{Eq:Specialization}
        \mathcal{L}_E(\xi;t)\mid_{t=\theta}=\log_E(\xi).
    \end{equation}
    
    In what follows, we provide a different construction of the deformation series \eqref{Eq:Deformation_of_Log}. Let $\mathcal{R}_0:=\Id_r\in\Mat_r(\mathbb{C}_\infty[t])$ and $\mathcal{R}_m:=\left(\Phi_E^{-1}\right)^{(m)}\dots\left(\Phi_E^{-1}\right)^{(1)}\in\Mat_r(\mathbb{C}_\infty[t])$ for each $m\geq 1$. For the convenience of later use, we set $\mathcal{B}_n(t)=0$ if $n<0$. Inspired by the calculation in \cite[Lem.~4.2]{KP21}, we apply the identities established in \cite[Lem.~6.12]{EGP14} to obtain an explicit formula for the first column of $\mathcal{R}_m$.
    
    \begin{lemma}\label{Lem:Formula}
        For $m\geq 0$ and $1\leq i\leq r$, we have
        \[
            [\mathcal{R}_m]_{i1}=\mathcal{B}_{m-(i-1)}(t).
        \]
    \end{lemma}
    
    \begin{proof}
        We prove this lemma by using the induction on $m$. Since $\mathcal{R}_0=\Id_r$ and $\mathcal{B}_0=1$, the desired result holds in the case $m=0$. For the general case, we recall that $\Phi_E$ is defined in \eqref{Eq:T_Frame_Of_Drinfeld_Modules}, and the direct calculation shows that
        \begin{equation}\label{Eq:Inverse_Of_Phi_E}
            \Phi_E^{-1}=\frac{1}{t-\theta}\begin{pmatrix}
            \kappa_1^{(-1)} & \kappa_2^{(-2)} & \cdots & \kappa_r^{(-r)} \\
            t-\theta & & & 0 \\
             & \ddots & & \vdots \\
             & & t-\theta & 0
            \end{pmatrix}\in\Mat_r(\mathbb{C}_\infty[t])\cap\GL_{r}(\mathbb{C}_\infty(t)).
        \end{equation}
        Since $\mathcal{R}_m=\left(\Phi_E^{-1}\right)^{(m)}\mathcal{R}_{m-1}$, we may apply \cite[Lem.~6.12]{EGP14} and the induction hypothesis to obtain
        \[
            [\mathcal{R}_m]_{11}=\sum_{j=1}^r[\left(\Phi_E^{-1}\right)^{(m)}]_{1j}[\mathcal{R}_{m-1}]_{j1}=\sum_{j=1}^r\frac{\kappa_j^{(m-j)}}{t-\theta^{q^m}}\mathcal{B}_{m-j}(t)=\mathcal{B}_m(t).
        \]
        The remaining cases $2\leq i\leq r$ follow immediately from the induction hypothesis and the direct calculation
        \[
            [\mathcal{R}_m]_{i1}=\sum_{j=1}^r[\left(\Phi_E^{-1}\right)^{(m)}]_{ij}[\mathcal{R}_{m-1}]_{j1}=[\mathcal{R}_{m-1}]_{(i-1)1}=\mathcal{B}_{m-(i-1)}(t).
        \]
    \end{proof}
    
    The following result is an immediate consequence of Lemma~\ref{Lem:Formula}, which provides a different expression of the deformation series given in \eqref{Eq:Deformation_of_Log}.
     
     \begin{proposition}\label{Prop:Connection}
        Let $\xi\in\mathbb{C}_\infty$ with $|\xi|_\infty<R_E$. Then
        \[
            \mathcal{L}_E(\xi;t)=\xi+\sum_{i=1}^\infty(\xi^{q^i},0,\dots,0)\left(\Phi_E^{-1}\right)^{(i)}\dots\left(\Phi_E^{-1}\right)^{(1)}(1,0,\dots,0)^\tr\in\TT_\theta.
        \]
     \end{proposition}
    
    \begin{proof}
        Note that $\mathcal{B}_0(t)=1$ and Lemma~\ref{Lem:Formula} shows that
        \[
            \left(\Phi_E^{-1}\right)^{(i)}\dots\left(\Phi_E^{-1}\right)^{(1)}(1,0,\dots,0)^\tr=(\mathcal{B}_{m}(t),\dots,\mathcal{B}_{m-(r-1)}(t)).
        \]
        It follows that
        \[
            \xi+\sum_{i=1}^\infty(\xi^{q^i},0,\dots,0)\left(\Phi_E^{-1}\right)^{(i)}\dots\left(\Phi_E^{-1}\right)^{(1)}(1,0,\dots,0)^\tr=\sum_{n=0}^\infty\mathcal{B}_n(t)\xi^{q^n}=\mathcal{L}_E(\xi;t).
        \]
    \end{proof}
    
    \begin{example}\label{Ex:Carlitz_Logarithm}
        Let $C=(\mathbb{G}_{a/\mathbb{C}_\infty},[\cdot])$ be the Carlitz module, which is a Drinfeld module of rank one with
        \begin{align*}
            [\cdot]:\mathbb{F}_q[t]&\to\mathbb{C}_\infty[\tau]\\
            t&\mapsto[t]:=\theta+\tau.
        \end{align*}
        Let $\mathcal{M}_C=\mathbb{C}_\infty[\tau]$ be the associated dual $t$-motive of $C$. Then $\Phi_C$ defined in \eqref{Eq:T_Frame_Of_Drinfeld_Modules} is given by $\Phi_C=(t-\theta)\in\mathbb{C}_\infty[t]$. As studied in \cite{Pap08}, if $\xi\in\mathbb{C}_\infty$ with $|\xi|_\infty<R_C=q^{q/(q-1)}$, then the deformation series of the Carlitz logarithm is given by
        \[
            \mathcal{L}_C(\xi;t)=\xi+\sum_{i=1}^\infty\frac{\xi^{q^i}}{(t-\theta^q)(t-\theta^{q^2})\cdots(t-\theta^{q^i})}\in\TT_\theta
        \]
        which coincides with Proposition~\ref{Prop:Connection}
        \[
            \mathcal{L}_C(\xi;t)=\xi+\sum_{i=1}^\infty\xi^{q^i}\left(\Phi_C^{-1}\right)^{(i)}\dots\left(\Phi_C^{-1}\right)^{(1)}.
        \]
    \end{example}
    
    \begin{remark}\label{Rem:Comparision_Anderson}
        In this remark, we explain how Proposition~\ref{Prop:Connection} is related to the more general theory due to Anderson. Let $E$ be a Drinfeld module of rank $r$ defined over $\mathbb{C}_\infty$ with the associated dual $t$-motive $\mathcal{M}_E$ and the $t$-frame $(\iota,\Phi_E)$ given in \eqref{Eq:T_Frame_Of_Drinfeld_Modules}. If we set $\mathscr{E}_1:=\epsilon_1\circ\iota$, then Anderson's exponentiation theorem (see \cite[Thm.~2.5.21]{HJ20}, \cite[Thm.~3.4.2]{NP21}) shows that for $\mathbf{h}\in\Mat_{1\times r}(\mathbb{C}_\infty[t])$ and $\mathbf{g}\in\Mat_{1\times r}(\TT_\theta)$ with
        \[
            \mathbf{g}^{(-1)}\Phi_E-\mathbf{g}=\mathbf{h},
        \]
        we have
        \[
            \exp_E(\mathscr{E}_0(\mathbf{g}+\mathbf{h}))=\mathscr{E}_1(\mathbf{h}).
        \]
        Now we identify $\mathbb{C}_\infty$ as a subset of $\Mat_{1\times r}(\mathbb{C}_\infty[t])$ by setting
        \begin{align*}
            \mathbb{C}_\infty&\to\Mat_{1\times r}(\mathbb{C}_\infty[t])\\
            \xi&\mapsto\mathbf{h}_\xi:=(\xi,0,\dots,0).
        \end{align*}
        For $\xi\in\mathbb{C}_\infty$ with sufficiently small $|\xi|_\infty$, we have
        \[
            \mathbf{g}_\xi:=\sum_{i=1}^\infty(\xi^{q^i},0,\dots,0)\left(\Phi_E^{-1}\right)^{(i)}\dots\left(\Phi_E^{-1}\right)^{(1)}\in\Mat_{1\times r}(\TT_\theta).
        \]
        One verifies directly that
        \[
            \mathbf{g}_\xi^{(-1)}\Phi_E-\mathbf{g}_\xi=\mathbf{h}_\xi.
        \]
        Hence by Anderson's exponentiation theorem we conclude that
        \[
            \exp_E(\mathscr{E}_0(\mathbf{g}_\xi+\mathbf{h}_\xi))=\mathscr{E}_1(\mathbf{h}_\xi)=\xi.
        \]
        Then \cite[Ex.~3.5.14]{NP21} provides a way to illustrate the expression
        \[
            \mathcal{L}_E(\xi;t)=(\mathbf{g}_\xi+\mathbf{h}_\xi)(1,0,\dots,0)^\tr
        \]
        given in Proposition~\ref{Prop:Connection}. The author is grateful to Matt Papanikolas for helpful discussion on this direction.
    \end{remark}
    
    \begin{remark}\label{Rem:Comparision_ANDTR}
        In what follows, we recall the notion of \emph{the inverse of the Frobenius} introduced in \cite[Section~2]{ANDTR20}. For our purpose, we restrict ourselves to the setting of Drinfeld modules, but in the meantime we also give a bit detailed description using the specific $t$-frame given in \eqref{Eq:T_Frame_Of_Drinfeld_Modules}. Let $P_0(t)=1$. For $j\in\mathbb{Z}_{>0}$, since $t-\theta^{q^j}$ is relatively prime to $t-\theta$ in $\mathbb{C}_\infty[t]$, we can choose $P_j(t)\in\mathbb{C}_\infty[t]$ which serves as the multiplicative inverse of $t-\theta^{q^j}$ in the quotient ring $\mathbb{C}_\infty[t]/(t-\theta)\mathbb{C}_\infty[t]$, that is, we have
        \[
            P_j(t)(t-\theta^{q^j})\equiv 1~(\mbox{mod}~(t-\theta)\mathbb{C}_\infty[t]).
        \]
        
        Let $E$ be a Drinfeld module of rank $r$ defined over $\mathbb{C}_\infty$ with the associated dual $t$-motive $\mathcal{M}_E$ and the $t$-frame $(\iota,\Phi_E)$ given in \eqref{Eq:T_Frame_Of_Drinfeld_Modules}. It has been shown in \eqref{Eq:Inverse_Of_Phi_E} that $(t-\theta)\Phi_E^{-1}\in\Mat_r(\mathbb{C}_\infty[t])$. Let $\mathbf{h}=(h_1,\dots,h_r)\in\Mat_{1\times r}(\mathbb{C}_\infty[t])$. Then for each $j\in\mathbb{Z}_{>0}$, one can verify that
        \begin{align*}
            \mathbf{y}_{\mathbf{h},j}:&=\mathbf{h}^{(j)}\left((t-\theta)\Phi_E^{-1}\right)^{(j)}\cdots\left((t-\theta)\Phi_E^{-1}\right)^{(1)}\\
            &=(t-\theta^{q^j})\cdots(t-\theta^q)\mathbf{h}^{(j)}\left(\Phi_E^{-1}\right)^{(j)}\cdots\left(\Phi_E^{-1}\right)^{(1)}
        \end{align*}
        is the unique element in $\Mat_{1\times r}(\mathbb{C}_\infty[t])$ satisfying
        \[
            \sigma^j\iota(\mathbf{y}_{\mathbf{h},j})=(t-\theta)(t-\theta^{q^{-1}})\cdots(t-\theta^{q^{1-j}})\iota(\mathbf{h}).
        \]
        If $\xi\in\mathbb{C}_\infty$, then we adopt the same notation as in the previous remark that $\mathbf{h}_\xi=(\xi,0,\dots,0)\in\Mat_{1\times r}(\mathbb{C}_\infty[t])$. The $j$-th inverse of the Frobenius is defined by
        \begin{align*}
            \varphi_j:\mathbb{C}_\infty&\to\mathbb{C}_\infty\\
            \xi&\mapsto\epsilon_0(P_0(t)\cdots P_j(t)\iota(\mathbf{y}_{\mathbf{h}_\xi,j})).
        \end{align*}
        We write $\log_E(X)=\sum_{i=0}^\infty Q_iX^{q^i}$. Then \cite[Prop.~2.2]{ANDTR20} shows that for each $j\in\mathbb{Z}_{>0}$
        \[
            \varphi_j(\xi)=Q_j\xi^{q^j}.
        \]
        In particular, suppose that $\xi\in\mathbb{C}_\infty$ with $|\xi|_\infty<R_E$, then
        \[
            \log_E(\xi)=\sum_{i=0}^\infty\varphi_i(\xi).
        \]
        
        In fact, Nathan Green pointed out to the author that the map $\epsilon_0$ can be extended to $\mathcal{M}_E\otimes_{\mathbb{C}_\infty[t]}\mathbb{C}_\infty[t]_{(t-\theta)}$, where $\mathbb{C}_\infty[t]_{(t-\theta)}$ is the localization of $\mathbb{C}_\infty[t]$ at the prime ideal $(t-\theta)\mathbb{C}_\infty[t]$ (see \cite[Prop.~5.6]{HJ20}). One can deduce immediately that
        \[
            \varphi_j(\xi)=(\epsilon_0\circ\iota)\left((\xi^{q^j},0,\dots,0)(\Phi_E^{-1})^{(j)}\cdots(\Phi_E^{-1})^{(1)}\right).
        \]
        This gives another way to understand our construction in Proposition~\ref{Prop:Connection}. Note that this idea has been used in the study of the logarithms of the iterated extensions of the Carlitz tensor powers (see \cite{CGM20}) and the log-algebraic identities for Anderson $t$-modules (see Step $1$ in the proof of \cite[Thm.~3.5]{GND20}).
    \end{remark}

\subsection{Furusho's analytic continuation}
    We introduce a generalization of Furusho's $\mathbb{F}_q[t]$-linear map 
    \begin{align*}
        \wp:\mathbb{T}&\to\mathbb{T}\\
        Z&\mapsto Z-Z^{(1)}
    \end{align*}
    which plays a crucial role in the analytic continuation. The proof of the following lemma is essentially the same as in \cite[Lem.~1.1.1]{Fur20}.
    \begin{lemma}[cf.~{\cite[Lem.~1.1.1]{Fur20}}]\label{Lem:Key_Lemma}
        Let 
        \begin{align*}
            \wp_r:\Mat_{1\times r}(\mathbb{T})&\to\Mat_{1\times r}(\mathbb{T})\\
            (Z_1,\dots,Z_r)&\mapsto(\wp(Z_1),\dots,\wp(Z_r))=(Z_1,\dots,Z_r)-(Z_1,\dots,Z_r)^{(1)}.
        \end{align*}
        Then
        \begin{enumerate}
            \item $\wp_r$ is surjective, and the preimage $\wp_r^{-1}(h_1,\dots,h_r)$ for each $(h_1,\dots,h_r)\in\Mat_{1\times r}(\mathbb{T})$ is given by $(h'_1,\dots,h'_r)+\mathbb{F}_q[t]^r$ for some $(h'_1,\dots,h'_r)\in\Mat_{1\times r}(\mathbb{T})$.
            \item For any $\mathbf{f}=(f_1,\dots,f_r)\in\Mat_{1\times r}(\mathbb{T})$, $\mathbf{f}$ and $\wp_r(\mathbf{f})$ have the same radius of convergence.
            \item If $V\subset\Mat_{1\times r}(\mathbb{T})$ is an $\mathbb{F}_q[t]$-submodule, then so is $\wp_r^{-1}(V)$.
            \item $\wp_r^{-1}(\Mat_{1\times r}(\mathbb{E}))=\Mat_{1\times r}(\mathbb{E})$.
        \end{enumerate}
    \end{lemma}
    
    For $(Z_1,\dots,Z_r)\in\Mat_{1\times r}(\mathbb{T})$, we introduce the series
    \[
        \mathscr{L}_0(Z_1,\dots,Z_r):=\sum_{i=0}^\infty(Z_1,\dots,Z_r)^{(i)}.
    \]
    Note that when $||(Z_1,\dots,Z_r)||<1$, it converges in $\Mat_{1\times r}(\TT)$ and it is $\mathbb{F}_q[t]$-linear with respect to $(Z_1,\dots,Z_r)$. Moreover, we have
    \[
        \wp_r(\mathscr{L}_0(Z_1,\dots,Z_r))=(Z_1,\dots,Z_r).
    \]
    Now Lemma~\ref{Lem:Key_Lemma} ensures that we can associate each $(Z_1,\dots,Z_r)\in\Mat_{1\times r}(\mathbb{T})$ with an element $\mathscr{L}_0(Z_1,\dots,Z_r)$ in the quotient $\mathbb{F}_q[t]$-submodule $\Mat_{1\times r}(\mathbb{T})/\Mat_{1\times r}(\mathbb{F}_q[t])$. Then, we deduce an extended $\mathbb{F}_q[t]$-linear map
    \begin{align*}
        \overset{\rightarrow}{\mathscr{L}_0}:\Mat_{1\times r}(\mathbb{T})&\to\Mat_{1\times r}(\mathbb{T})/\Mat_{1\times r}(\mathbb{F}_q[t])\\
        (Z_1,\dots,Z_r)&\mapsto\wp_r^{-1}(Z_1,\dots,Z_r).
    \end{align*}
    By following the terminology in \cite{Fur20}, we call $\mathscr{L}_0^\circ:\Mat_{1\times r}(\mathbb{T})\to\Mat_{1\times r}(\mathbb{T})$ a branch of $\overset{\rightarrow}{\mathscr{L}_0}$ whenever $\mathscr{L}_0^\circ$ is an $\mathbb{F}_q$-linear lift of $\overset{\rightarrow}{\mathscr{L}_0}$.
    
    Let $E$ be a Drinfeld module of rank $r$ defined over $\mathbb{C}_\infty$. Recall that it has the associated dual $t$-motive $\mathcal{M}_E$ with the $t$-frame $(\iota,\Phi_E)$ given in \eqref{Eq:T_Frame_Of_Drinfeld_Modules}.
    Recall that it has a rigid analytic trivialization $\Psi_E\in\GL_r(\mathbb{T}_\theta)\cap\Mat_r(\mathbb{E})\subset\Mat_r(\TT)$ satisfying the Frobenius difference equation
    \[
        \Psi_E^{(-1)} = \Phi_E\Psi_E.
    \]
    Let $\{\mathbf{e}_j\}_{j=1}^r$ be the standard $\TT$-basis of $\Mat_{1\times r}(\TT)$. For $Z\in\mathbb{T}$, we introduce the series
    \[
        \mathscr{L}_E(Z):=\mathscr{L}_0(Z\mathbf{e}_1\Psi_E)\Psi_E^{-1}\mathbf{e}_1^\tr.
    \]
    Note that when $||Z\mathbf{e}_1\Psi_E||<1$, it converges in $\TT$ and it is $\mathbb{F}_q[t]$-linear with respect to $Z$.
    The series $\mathscr{L}_E(Z)$ can be formally expressed as follows:
    \begin{align*}
        \mathscr{L}_E(Z)&=\mathscr{L}_0(Z\mathbf{e}_1\Psi_E)\Psi_E^{-1}\mathbf{e}_1^\tr\\
        &=\sum_{i=0}^\infty Z^{(i)}\mathbf{e}_1\Psi_E^{(i)}\Psi_E^{-1}\mathbf{e}_1^\tr\\
        &=Z+\sum_{i=1}^\infty(Z^{(i)},0,\dots,0)\left(\Phi_E^{-1}\right)^{(i)}\dots\left(\Phi_E^{-1}\right)^{(1)}(1,0,\dots,0)^\tr
    \end{align*}
    which agrees with the expression of Proposition~\ref{Prop:Connection}. In particular, when $\xi\in\mathbb{C}_\infty$ with $||\xi\mathbf{e}_1\Psi_E||<1$ and $|\xi|_\infty<R_E$, we have $\mathscr{L}(\xi)=\mathcal{L}_E(\xi;t)\in\TT$.
    
    Now, we define the following $\mathbb{F}_q[t]$-linear map, which can be regarded as the continuation of the deformation series of the Drinfeld logarithm. We set
    \[
        \overset{\rightarrow}{\mathscr{L}_E}:\mathbb{T}\to\mathbb{T}/\left(\Mat_{1\times r}(\mathbb{F}_q[t])\Psi_E^{-1}\mathbf{e}_1^\tr\right)
    \]
    by sending $Z\in\mathbb{T}$ to
    \[
        \overset{\rightarrow}{\mathscr{L}_E}(Z):=\overset{\rightarrow}{\mathscr{L}_0}(Z\mathbf{e}_1\Psi_E)\Psi_E^{-1}\mathbf{e}_1^\tr.
    \]
    By following the terminology in \cite{Fur20}, we call $\mathscr{L}_E^\circ:\mathbb{T}\to\mathbb{T}$ a branch of $\overset{\rightarrow}{\mathscr{L}_E}$ whenever $\mathscr{L}_E^\circ$ is an $\mathbb{F}_q$-linear lift of $\overset{\rightarrow}{\mathscr{L}_E}$.

    In what follows, we carry out the analytic continuation of Drinfeld logarithms by the specialization $t=\theta$. Recall from \cite[Ex.~3.5.14]{NP21} that for $\mathbf{a}=(\alpha_1,\dots,\alpha_r)\in\Mat_{1\times r}(\mathbb{T}_\theta)$ we have
    \[
        \mathscr{E}_0(\mathbf{a})=(\mathbf{a}\mathbf{e}_1^\tr)\mid_{t=\theta}=\alpha_1(\theta).
    \]
    In particular, by \eqref{Eq:E_0_Periods} we have
    \[
        \Lambda_E=\mathscr{E}_0(\Mat_{1\times r}(\mathbb{F}_q[t])\Psi_E^{-1})=\left(\Mat_{1\times r}(\mathbb{F}_q[t])\Psi_E^{-1}\mathbf{e}_1^\tr\right)\mid_{t=\theta}\subset\mathbb{C}_\infty.
    \]
    Let $\xi\in\mathbb{C}_\infty$. By Lemma~\ref{Lem:Key_Lemma}, we see that $t=\theta$ is contained in the radius of the convergence of each entry of any representative in the coset
    \[
        \overset{\rightarrow}{\mathscr{L}_0}(\xi\mathbf{e}_1\Psi_E)=\wp_r^{-1}(\xi\mathbf{e}_1\Psi_E).
    \]
    Since $\Psi_E\in\GL_r(\TT_\theta)$, we conclude that $\overset{\rightarrow}{\mathscr{L}}_E(\xi)=\overset{\rightarrow}{\mathscr{L}_0}(\xi\mathbf{e}_1\Psi_E)\Psi_E^{-1}\mathbf{e}_1^\tr$ is well-defined at $t=\theta$.
    Thus, we can define the following $\mathbb{F}_q$-linear map
    \begin{align*}
        \overset{\rightarrow}{\log_E}:\mathbb{C}_\infty&\to\mathbb{C}_\infty/\Lambda_E\\
        \xi&\mapsto\overset{\rightarrow}{\mathscr{L}}_E(\xi)\mid_{t=\theta}.
    \end{align*}
    A branch $\log_E^\circ:\mathbb{C}_\infty\to\mathbb{C}_\infty$ means an $\mathbb{F}_q$-linear lift of $\overset{\rightarrow}{\log_E}$.
    
    \begin{example}
        Let $C=(\mathbb{G}_{a/\mathbb{C}_\infty},[\cdot])$ be the Carlitz module given in Example~\ref{Ex:Carlitz_Logarithm}. Recall that the $t$-frame $(\iota,\Phi_C)$ of the associated dual $t$-motive $\mathcal{M}_C$ is given by $\Phi_C=(t-\theta)\in\mathbb{C}_\infty[t]$. It has a rigid analytic trivialization \begin{equation}\label{Eq:Anderson_Thakur_Series}
            \Omega(t):=(-\theta)^{\frac{-q}{q-1}}\prod_{i=1}^\infty\left(1-\frac{t}{\theta^{q^i}}\right)\in\mathbb{E},
        \end{equation}
        where $(-\theta)^{1/(q-1)}$ is a fixed choice of the $(q-1)$-st root of $(-\theta)$. In what follows, we provide a specific choice of the representative $\log_C^\circ(\alpha)$ of $\overset{\rightarrow}{\log_C}(\alpha)\in\mathbb{C}_\infty/\Lambda_C$ so that $\log_C^\circ(\alpha)=\log_C(\alpha)$ whenever $\alpha\in k_\infty$ with $|\alpha|_\infty<q^{q/(q-1)}$.
        
        Recall that
        \[
            \wp^{-1}(\alpha\Omega)=\{f=\sum_{i\geq 0}f_it^i\in\TT\mid f-f^{(1)}=\alpha\Omega\}\subset\mathbb{E}.
        \]
        We aim to find an element $f=\sum_{i\geq 0}f_it^i\in\wp^{-1}(\alpha\Omega)$ so that $f/\Omega\in k_\infty\llbracket t\rrbracket\cap\TT_\theta$. In particular, $\log_C^\circ(\alpha):=(f/\Omega)\mid_{t=\theta}$ gives a representative of lifts of $\overset{\rightarrow}{\log_C}(\alpha)\in\mathbb{C}_\infty/\Lambda_C$ in $k_\infty$. Then the fact
        \[
            \left(\log_C(\alpha)+\Lambda_C\right)\cap k_\infty=\{\log_C(\alpha)\}
        \]
        implies that $\log_C^\circ(\alpha)=\log_C(\alpha)$. The rest of this example will be occupied by constructing such $f=\sum_{i\geq 0}f_it^i\in\wp^{-1}(\alpha\Omega)$ so that $f/\Omega\in k_\infty\llbracket t\rrbracket\cap\TT_\theta$.
        
        We begin by noticing that if we express $\Omega(t)=\sum_{i\geq 0}a_it^i$, then the Frobenius difference equation $\Omega^{(-1)}=(t-\theta)\Omega$ implies that
        \begin{equation}\label{Eq:recursive}
            a_i+\theta^qa_i^q=a_{i-1}^q,~i\geq 0.
        \end{equation}
        Here we adopt the convention $a_{-1}:=0$. Since
        \[
            \Omega(t)=(-\theta)^{\frac{-q}{q-1}}\prod_{i=1}^\infty\left(1-\frac{t}{\theta^{q^i}}\right),
        \]
        we have $a_0=(-\theta)^{\frac{-q}{q-1}}$ and $||\Omega||=|a_0|_\infty=q^{-q/(q-1)}$, which implies that $a_i/a_0\in k_\infty$ and $|a_i/a_0|_\infty\leq 1$. To find the desired $f=\sum_{i\geq 0}f_it^i\in\wp^{-1}(\alpha\Omega)$, it is enough to determine $f_i\in\mathbb{C}_\infty$ so that $f_i/a_0\in k_\infty$. For $\alpha\in k_\infty$ with $|\alpha|_\infty<q^{q/(q-1)}$ and $i\in\mathbb{Z}_{\geq 0}$, the requirement $f\in\wp^{-1}(\alpha\Omega)$ enforces that
        \begin{equation}\label{Eq:Solving_Coefficients_1}
            f_i-f_i^q=\alpha a_i.
        \end{equation}
        By \eqref{Eq:recursive}, we know that $a_0=-\theta^qa_0^q$, and thus \eqref{Eq:Solving_Coefficients_1} becomes
        \begin{equation}\label{Eq:Solving_Coefficients_2}
            \left(\frac{f_i}{a_0}\right)^q+\theta^q\left(\frac{f_i}{a_0}\right)-\alpha\theta^q\left(\frac{a_i}{a_0}\right)=0.
        \end{equation}
        We claim that the equation $Y_i^q+\theta^qY_i-\alpha\theta^q(a_i/a_0)=0$ admits a root in $k_\infty$. To see this claim, let $\mathrm{ord}_\infty(\cdot)$ be the normalized valuation on $k_\infty$ so that $\mathrm{ord}_\infty(\theta)=-1$. Then we have $\mathrm{ord}_\infty(\alpha\theta^q(a_i/a_0))\in\mathbb{Z}$ and
        \[
            \mathrm{ord}_\infty(\alpha\theta^q(a_i/a_0))=\mathrm{ord}_\infty(\alpha)-q+\mathrm{ord}_\infty(a_i/a_0)\geq -1-q.
        \]
        Therefore, the Newton polygon of $Y_i^q+\theta^qY_i-\alpha\theta^q(a_i/a_0)=0$ consists of three points $(0,\mathrm{ord}_\infty(\alpha\theta^q(a_i/a_0))),~(1,-q),~(q,0)$, which always has an edge of length one with integer slope. This leads to the claim and the desired coefficients $f_i$.
        
        We mention that this example gives a specific choice of \cite[Rem.~1.3.6]{Fur20} in the case of $n=1$ together with some explicit analysis. Moreover, we demonstrate how this method can be used to construct some non-zero elements in $\exp_C^{-1}(k_\infty)\cap k_\infty$ without knowing the logarithmic series $\log_C(X)\in k_\infty\llbracket X\rrbracket$ precisely.
    \end{example}
    
    The following theorem ensures that $\overset{\rightarrow}{\log_E}$ is an analytic continuation of $\log_E$.
    \begin{theorem}[cf.~{\cite[Prop.~1.3.4]{Fur20}}]\label{Thm:Main_Thm_1}
        Let $\xi\in\mathbb{C}_\infty$. Then the following assertions hold.
        \begin{enumerate}
            \item If $|\xi|_\infty<R_E$, then $\overset{\rightarrow}{\log_E}(\xi)\equiv\log_E(\xi)~(\mathrm{mod}~\Lambda_E)$.
            \item $\overset{\rightarrow}{\log_E}$ locally admits an analytic lifts, that is, for any point $\xi\in\mathbb{C}_\infty$, there is a closed disk $\mathbb{D}\subset\mathbb{C}_\infty$ with center $\xi$ and a branch $\log_E^\circ$ so that $\log_E^\circ\mid_{\mathbb{D}}:\mathbb{D}\to\mathbb{C}_\infty$ is induced by a power series.
        \end{enumerate}
    \end{theorem}
    
    \begin{proof}
        The first part follows from our construction together with Proposition~\ref{Prop:Connection} and \eqref{Eq:Specialization}. To see the second part, we note from our construction that $\overset{\rightarrow}{\log_E}$ is additive. Furthermore, we have
        \[
            \overset{\rightarrow}{\log_E}(\xi)\equiv\log_E(\xi)~(\mathrm{mod}~\Lambda_E)
        \]
        by the first part whenever $\xi\in\mathbb{C}_\infty$ with $|\xi|_\infty<R_E$. Then the desired assertion follows from the fact that $\log_E$ is induced by a power series converging on the closed disk of radius smaller than $R_E$ and the period lattice is discrete inside $\mathbb{C}_\infty$.
    \end{proof}
    
    One of the most important features of the Drinfeld logarithm is the compatibility with the action coming from the Drinfeld module, namely, the functional equation
    \[
        \log_E\circ\phi_t=\partial\phi_t\circ\log_E.
    \]
    Our next result asserts that our extended Drinfeld logarithm satisfies the same functional equation.
    
    \begin{theorem}\label{Thm:Functional_Equation_Of_Extended_Logrithm}\label{Thm:Main_Thm_2}
        Let $\xi\in\mathbb{C}_\infty$. Then, we have
        \[
            \overset{\rightarrow}{\log_E}(\phi_t(\xi))\equiv\theta\overset{\rightarrow}{\log_E}(\xi) ~(\mathrm{mod}~\Lambda_E).
        \]
    \end{theorem}
    
    To prove Theorem~\ref{Thm:Functional_Equation_Of_Extended_Logrithm}, we need some preliminary calculations. For $(Z_1,\dots,Z_r)\in\Mat_{1\times r}(\mathbb{T})$, we recall that we have
    \begin{align*}
        \wp_r((Z_1,\dots,Z_r)^{(1)})&=(Z_1,\dots,Z_r)^{(1)}-(Z_1,\dots,Z_r)^{(2)}\\
        &=\left((Z_1,\dots,Z_r)-(Z_1,\dots,Z_r)^{(1)}\right)^{(1)}\\
        &=\wp_r((Z_1,\dots,Z_r))^{(1)}.
    \end{align*}
    Then, $\wp_r(\overset{\rightarrow}{\mathscr{L}_0}(Z_1,\dots,Z_r))=(Z_1,\dots,Z_r)$ implies that
    \begin{equation}\label{Eq:Property_Of_L0}
        \overset{\rightarrow}{\mathscr{L}_0}((Z_1,\dots,Z_r)^{(1)})=\overset{\rightarrow}{\mathscr{L}_0}((Z_1,\dots,Z_r))^{(1)}.
    \end{equation}
    
    \begin{lemma}\label{Lem:Preliminary_Calculations}
        Let $Z\in\mathbb{T}$. Then, for $2\leq j\leq r$, we have
        \[
            \overset{\rightarrow}{\mathscr{L}_0}\left(Z\mathbf{e}_j\Psi_E\right)\equiv\overset{\rightarrow}{\mathscr{L}_0}\left(Z^{(j-1)}\mathbf{e}_{1}\Psi_E\right)+Z^{(j-2)}\mathbf{e}_2\Psi_E+\cdots+Z\mathbf{e}_j\Psi_E~(\mathrm{mod}~\Mat_{1\times r}(\mathbb{F}_q[t])).
        \]
    \end{lemma}
    
    \begin{proof}
        Since $\Psi_E=\Phi_E^{-1}\Psi_E^{(-1)}$, we have
        \begin{equation}\label{Eq:Recursive_Formula}
            Z\mathbf{e}_j\Psi_E=Z\mathbf{e}_j\Phi_E^{-1}\Psi_E^{(-1)}=\left(Z^{(1)}\mathbf{e}_{j-1}\Psi_E\right)^{(-1)}.
        \end{equation}
        Here the second equality uses the fact that $\mathbf{e}_j\Phi_E^{-1}=\mathbf{e}_{j-1}^{(-1)}$ which is a consequence of \eqref{Eq:Inverse_Of_Phi_E}.
        On the one hand, we have
        \[
            \wp_r\left(\overset{\rightarrow}{\mathscr{L}_0}(Z\mathbf{e}_j\Psi_E)\right)=Z\mathbf{e}_j\Psi_E.
        \]
        On the other hand, by \eqref{Eq:Property_Of_L0} and \eqref{Eq:Recursive_Formula} we have
        \begin{align*}
            \wp_r\left(\overset{\rightarrow}{\mathscr{L}_0}(Z\mathbf{e}_j\Psi_E)\right)=\overset{\rightarrow}{\mathscr{L}_0}(Z\mathbf{e}_j\Psi_E)-\overset{\rightarrow}{\mathscr{L}_0}(Z^{(1)}\mathbf{e}_{j-1}\Psi_E).
        \end{align*}
        In particular, we deduce that
        \[
            \overset{\rightarrow}{\mathscr{L}_0}(Z\mathbf{e}_j\Psi_E)\equiv\overset{\rightarrow}{\mathscr{L}_0}(Z^{(1)}\mathbf{e}_{j-1}\Psi_E)+Z\mathbf{e}_j\Psi_E~(\mathrm{mod}~\Mat_{1\times r}(\mathbb{F}_q[t])).
        \]
        Finally, we apply the above equality inductively and eventually obtain
        \[
            \overset{\rightarrow}{\mathscr{L}_0}\left(Z\mathbf{e}_j\Psi_E\right)\equiv\overset{\rightarrow}{\mathscr{L}_0}\left(Z^{(j-1)}\mathbf{e}_{1}\Psi_E\right)+Z^{(j-2)}\mathbf{e}_2\Psi_E+\cdots+Z\mathbf{e}_j\Psi_E~(\mathrm{mod}~\Mat_{1\times r}(\mathbb{F}_q[t])).
        \]
    \end{proof}
    
    The next lemma can be seen as an extended version of \cite[Thm.~6.13(e)]{EGP14}.
    
    \begin{lemma}\label{Lem:Lem:Formula_For_T_Deformation_Logarithm}
        Let $Z\in\mathbb{T}$. Then, we have
        \begin{align*}
            t\overset{\rightarrow}{\mathscr{L}_0}\left(Z\mathbf{e}_1\Psi_E\right)&\equiv\overset{\rightarrow}{\mathscr{L}_0}\left(\phi_t(Z)\mathbf{e}_1\Psi_E\right)+(t-\theta)Z\mathbf{e}_1\Psi_E\\
            &+\sum_{j=2}^r\sum_{s=1}^{j-1}\kappa_j^{(-s)}Z^{(j-s)}\mathbf{e}_{s+1}\Psi_E~(\mathrm{mod}~\Mat_{1\times r}(\mathbb{F}_q[t])).
        \end{align*}
        Here, for $\sum_{i=0}^sa_i\tau^i\in\overline{k}[\tau]$ we set $\sum_{i=0}^sa_i\tau^i\left(Z\right):=\sum_{i=0}^sa_iZ^{(i)}\in\mathbb{T}$.
    \end{lemma}
    
    \begin{proof}
        Since $t=\theta+(t-\theta)$, we have
        \begin{equation}\label{Eq:First_Eq}
            t\overset{\rightarrow}{\mathscr{L}_0}\left(Z\mathbf{e}_1\Psi_E\right)=\overset{\rightarrow}{\mathscr{L}_0}\left(tZ\mathbf{e}_1\Psi_E\right)=\overset{\rightarrow}{\mathscr{L}_0}\left(\theta Z\mathbf{e}_1\Psi_E\right)+\overset{\rightarrow}{\mathscr{L}_0}\left((t-\theta)Z\mathbf{e}_1\Psi_E\right).
        \end{equation}
        Since $\Psi_E=\Phi_E^{-1}\Psi_E^{(-1)}$, \eqref{Eq:Inverse_Of_Phi_E} implies that
        \begin{equation}\label{Eq:Eq_2}
            (t-\theta)Z\mathbf{e}_1\Psi_E=\left((\kappa_1Z^{(1)},\dots,\kappa_r^{(1-r)}Z^{(1)})\Psi_E\right)^{(-1)}.
        \end{equation}
        On the one hand, we have
        \[
            \wp_r\left(\overset{\rightarrow}{\mathscr{L}_0}((t-\theta)Z\mathbf{e}_1\Psi_E)\right)=(t-\theta)Z\mathbf{e}_1\Psi_E.
        \]
        On the other hand, by \eqref{Eq:Property_Of_L0} and \eqref{Eq:Eq_2} we have
        \[
            \wp_r\left(\overset{\rightarrow}{\mathscr{L}_0}((t-\theta)Z\mathbf{e}_1\Psi_E)\right)=\overset{\rightarrow}{\mathscr{L}_0}((t-\theta)Z\mathbf{e}_1\Psi_E)-\overset{\rightarrow}{\mathscr{L}_0}\left((\kappa_1Z^{(1)},\dots,\kappa_r^{(1-r)}Z^{(1)})\Psi_E\right).
        \]
        In particular, we deduce that
        \begin{align*}
            \overset{\rightarrow}{\mathscr{L}_0}((t-\theta)Z\mathbf{e}_1\Psi_E)\equiv(t-\theta)Z\mathbf{e}_1\Psi_E+\sum_{j=1}^r\overset{\rightarrow}{\mathscr{L}_0}\left(\kappa_j^{(1-j)}Z^{(1)}\mathbf{e}_j\Psi_E\right)~(\mathrm{mod}~\Mat_{1\times r}(\mathbb{F}_q[t])).
        \end{align*}
        For $2\leq j\leq r$, we apply Lemma~\ref{Lem:Preliminary_Calculations} to $\overset{\rightarrow}{\mathscr{L}_0}\left(\kappa_j^{(1-j)}Z^{(1)}\mathbf{e}_j\Psi_E\right)$. Then, we obtain
        \[
            \overset{\rightarrow}{\mathscr{L}_0}\left(\kappa_j^{(1-j)}Z^{(1)}\mathbf{e}_j\Psi_E\right)\equiv\overset{\rightarrow}{\mathscr{L}_0}\left(\kappa_jZ^{(j)}\mathbf{e}_{1}\Psi_E\right)+\sum_{s=1}^{j-1}\kappa_j^{(-s)}Z^{(j-s)}\mathbf{e}_{s+1}\Psi_E~(\mathrm{mod}~\Mat_{1\times r}(\mathbb{F}_q[t])).
        \]
        Then, \eqref{Eq:First_Eq} becomes
        \begin{align*}
            t\overset{\rightarrow}{\mathscr{L}_0}\left(Z\mathbf{e}_1\Psi_E\right)&\equiv\overset{\rightarrow}{\mathscr{L}_0}\left(\phi_t(Z)\mathbf{e}_1\Psi_E\right)+(t-\theta)Z\mathbf{e}_1\Psi_E\\
            &+\sum_{j=2}^r\sum_{s=1}^{j-1}\kappa_j^{(-s)}Z^{(j-s)}\mathbf{e}_{s+1}\Psi_E~(\mathrm{mod}~\Mat_{1\times r}(\mathbb{F}_q[t])).
        \end{align*}
        as desired.
    \end{proof}
    {\textbf{Proof of Theorem~\ref{Thm:Functional_Equation_Of_Extended_Logrithm}}:~}
    For $Z\in\mathbb{T}$, we have
    \[
        t\overset{\rightarrow}{\mathscr{L}_E}(Z)=t\overset{\rightarrow}{\mathscr{L}_0}\left(Z\mathbf{e}_1\Psi_E\right)\Psi_E^{-1}\mathbf{e}_1^\mathrm{tr}.
    \]
    By Lemma~\ref{Lem:Lem:Formula_For_T_Deformation_Logarithm}, the equality above becomes
    \[
        t\overset{\rightarrow}{\mathscr{L}_E}(Z)\equiv\overset{\rightarrow}{\mathscr{L}_0}\left(\phi_t(Z)\mathbf{e}_1\Psi_E\right)\Psi_E^{-1}\mathbf{e}_1^\mathrm{tr}+(t-\theta)Z\mathbf{e}_1^\mathrm{tr}~(\mathrm{mod}~\Mat_{1\times r}(\mathbb{F}_q[t])\Psi_E^{-1}).
    \]
    Here, we use the fact that $\mathbf{e}_{j}\mathbf{e}_1^\mathrm{tr}=0$ for each $2\leq j\leq r$.
    Now we consider $\xi\in\mathbb{C}_\infty$. Then, we have
    \begin{align*}
        \theta\overset{\rightarrow}{\log_E}(\xi)&\equiv t\overset{\rightarrow}{\mathscr{L}_E}(\xi)\mid_{t=\theta}~(\mathrm{mod}~\Lambda_E)\\
        &\equiv\left(\overset{\rightarrow}{\mathscr{L}_0}\left(\phi_t(\xi)\mathbf{e}_1\Psi_E\right)\Psi_E^{-1}\mathbf{e}_1^\mathrm{tr}+(t-\theta)\xi\right)\mid_{t=\theta}~(\mathrm{mod}~\Lambda_E)\\
        &\equiv\overset{\rightarrow}{\log_E}(\phi_t(\xi)) ~(\mathrm{mod}~\Lambda_E).
    \end{align*}
    as desired.\qed
        
    \begin{corollary}\label{Thm:Main_Thm_3}
        We denote by $\overset{\rightarrow}{\exp_E}:\mathbb{C}_\infty/\Lambda_E\to\mathbb{C}_\infty$ the induced map from the Drinfeld exponential function $\exp_E:\mathbb{C}_\infty\to\mathbb{C}_\infty$. Then, $\overset{\rightarrow}{\log_E}:\mathbb{C}_\infty\to\mathbb{C}_\infty/\Lambda_E$ is the inverse of $\overset{\rightarrow}{\exp_E}$.
    \end{corollary}
    
    \begin{proof}
        The proof follows the same spirit as the proof of \cite[Thm.~2.1.3.(2)]{Fur20}. Let $\xi\in\mathbb{C}_\infty$. Consider the sequence $\xi_n:=\xi/\theta^n$ for each $n\in\mathbb{Z}_{>0}$. Since $\lim_{n\to\infty}\xi_n=0$ and $\exp_E$ are continuous, there exists $N>0$ such that $|\exp_E(\xi_n)|_\infty<R_E$ for each $n>N$. Then, for a fixed $n>N$, we have
        \begin{align*}
            \overset{\rightarrow}{\log_E}\circ\exp_E(\xi)&=\overset{\rightarrow}{\log_E}\circ\exp_E(\theta^n\xi_n)\\
            &=\overset{\rightarrow}{\log_E}\circ\phi_{t^n}\circ\exp_E(\xi_n)\\
            &=\partial\phi_{t^n}\circ\overset{\rightarrow}{\log_E}\circ\exp_E(\xi_n)\\
            &=\theta^n\log_E(\exp_E(\xi_n))\\
            &=\xi.
        \end{align*}
        Since $\exp_E:\mathbb{C}_\infty\to\mathbb{C}_\infty$ is surjective with kernel $\Lambda_E$, we conclude that $\overset{\rightarrow}{\log_E}$ is the inverse of $\overset{\rightarrow}{\exp_E}$.
    \end{proof}

\bibliographystyle{alpha}

\end{document}